\documentclass[12pt]{amsart}
\usepackage{amscd,amssymb,amsmath,multicol}
\newtheorem{thm}[equation]{Theorem}
\numberwithin{equation}{section}
\newtheorem{cor}[equation]{Corollary}

\newtheorem{lem}[equation]{Lemma}

\begin{document}
\raggedbottom \voffset=-.7truein \hoffset=0truein \vsize=8truein
\hsize=6truein \textheight=8truein \textwidth=6truein
\baselineskip=18truept
\def\mapright#1{\ \smash{\mathop{\longrightarrow}\limits^{#1}}\ }
\def\ss{\smallskip}
\def\ssum{\sum\limits}
\def\dsum{{\displaystyle{\sum}}}
\def\la{\langle}
\def\ra{\rangle}
\def\on{\operatorname}
\def\u{\on{u}}
\def\lg{\on{lg}}
\def\a{\alpha}
\def\bz{{\Bbb Z}}
\def\eps{\epsilon}
\def\br{{\bold R}}
\def\bc{{\bold C}}
\def\bN{{\bold N}}
\def\nut{\widetilde{\nu}}
\def\tfrac{\textstyle\frac}
\def\product{\prod}
\def\b{\beta}
\def\G{\Gamma}
\def\g{\gamma}
\def\zt{{\Bbb Z}_2}
\def\zth{{\bold Z}_2^\wedge}
\def\bs{{\bold s}}
\def\bg{{\bold g}}
\def\bof{{\bold f}}
\def\bq{{\bold Q}}
\def\be{{\bold e}}
\def\line{\rule{.6in}{.6pt}}
\def\xb{{\overline x}}
\def\xbar{{\overline x}}
\def\ybar{{\overline y}}
\def\zbar{{\overline z}}
\def\ebar{{\overline \be}}
\def\nbar{{\overline n}}
\def\fbar{{\overline f}}
\def\Ubar{{\overline U}}
\def\et{{\widetilde e}}
\def\ni{\noindent}
\def\ms{\medskip}
\def\ahat{{\hat a}}
\def\bhat{{\hat b}}
\def\chat{{\hat c}}
\def\nbar{{\overline{n}}}
\def\minp{\min\nolimits'}
\def\N{{\Bbb N}}
\def\Z{{\Bbb Z}}
\def\Q{{\Bbb Q}}
\def\R{{\Bbb R}}
\def\C{{\Bbb C}}
\def\el{\ell}
\def\mo{\on{mod}}
\def\dstyle{\displaystyle}
\def\ds{\dstyle}
\def\Remark{\noindent{\it  Remark}}
\title
{Binomial coefficients involving infinite powers of primes}
\author{Donald M. Davis}
\address{Department of Mathematics, Lehigh University\\Bethlehem, PA 18015, USA}
\email{dmd1@lehigh.edu}
\date{January 26, 2013}


\maketitle
\begin{abstract} If $p$ is a prime and $n$ a positive integer, let $\nu_p(n)$ denote the exponent of $p$ in $n$, and $\u_p(n)=n/p^{\nu_p(n)}$ the unit part of $n$.
If $\a$ is a positive integer not divisible by $p$, we show that the $p$-adic limit of $(-1)^{p\a e}\u_p((\a p^e)!)$ as $e\to\infty$ is a well-defined $p$-adic integer, which
we call $z_{\a,p}$. In terms of these, we then give a formula for the $p$-adic limit of $\binom{ap^e+c}{bp^e+d}$ as $e\to\infty$, which we call $\binom{ap^\infty+c}{bp^\infty+d}$.
Here $a\ge b$ are positive integers, and $c$ and $d$ are integers.
\end{abstract}
\section{Statement of results}\label{intro}
Let $p$ be a prime number, fixed throughout.
The set $\Z_p$ of $p$-adic integers consists of expressions of the form $x=\ds\sum_{i=0}^\infty c_ip^i$ with $0\le c_i\le p-1$.
The nonnegative integers are those $x$ for which the sum is finite.
The metric on $\Z_p$ is defined by $d(x,y)=1/p^{\nu(x-y)}$, where $\nu(x)=\min\{i:\ c_i\ne0\}$. (See, e.g., \cite{Gou}).) The prime $p$ will be implicit in most of our notation.

If $n$ is a positive integer, let $\u(n)=n/p^{\nu(n)}$ denote the unit part of $n$ (with respect to $p$). Our first result is
\begin{thm}\label{thm1} Let $\a$ be a positive integer which is not divisible by $p$. If $p^e>4$, then
$$\u((\a p^{e-1})!)\equiv(-1)^{p\a}\u((\a p^e)!)\mod p^e.$$
\end{thm}
\begin{cor}\label{cordef} If $\a$ is as in Theorem \ref{thm1}, then $\ds\lim_{e\to\infty}(-1)^{p\a e}\u((\a p^e)!)$ exists in $\Z_p$. We denote this
limiting $p$-adic integer by $z_\a$.\end{cor}

If $p=2$ or $\a$ is even, then $z_\a$ could be thought of as $\u((\a p^\infty)!)$.
It is easy for  {\tt Maple} to compute $z_\a$ mod $p^m$ for $m$ fairly large. For example, if $p=2$, then $z_1\equiv 1+2+2^3+2^7+2^9+2^{10}+2^{12}\mod 2^{15}$,
and if $p=3$, then $z_1\equiv1+2\cdot3+2\cdot3^2+2\cdot3^4+3^6+2\cdot3^7+2\cdot3^8\mod 3^{11}$. It would be interesting to know
if there are algebraic relationships among the various $z_\a$ for a fixed prime $p$.

There are two well-known formulas for the power of $p$ dividing a binomial coefficient $\binom ab$. (See, e.g., \cite{Gr}). One is that
$$\nu\tbinom ab=\tfrac1{p-1}(d_p(b)+d_p(a-b)-d_p(a)),$$
where $d_p(n)$ denotes sum of the coefficients when $n$ is written in $p$-adic form as above. Another is that $\nu\binom ab$ equals
the number of carries in the base-$p$ addition of $b$ and $a-b$.
Clearly $\nu\binom{ap^e}{bp^e}=\nu\binom ab$.

Our next result involves the unit factor of $\binom{ap^e}{bp^e}$. Here one of $a$ or $b$ might be divisible by $p$.
For a positive integer $n$, let $z_n=z_{\u(n)}$, where $z_{\u(n)}\in\Z_p$ is as defined in Corollary \ref{cordef}.

\begin{thm} \label{thm2} Suppose $1\le b\le a$, $\nu(a-b)=0$, and $\{\nu(a),\nu(b)\}=\{0,k\}$ with $k\ge0$. Then
$$\u\binom{ap^e}{bp^e}\equiv(-1)^{pck}\frac{z_a}{z_bz_{a-b}}\mod p^e,$$
where $c=\begin{cases}a&\text{if }\nu(a)=k\\ b&\text{if }\nu(b)=k.\end{cases}$
\end{thm}

Since $\nu\binom{ap^e}{bp^e}$ is independent of $e$, we obtain the following immediate corollary.
\begin{cor} In the notation and hypotheses of Theorem \ref{thm2},  in $\Z_p$
$$\binom{ap^\infty}{bp^\infty}:=\lim_{e\to\infty}\binom{ap^e}{bp^e}=p^{\nu\binom ab}(-1)^{pck}\frac{z_a}{z_bz_{a-b}}.$$\label{cor2}
\end{cor}

Our final result analyzes $\binom{ap^\infty+c}{bp^\infty+d}$, where $c$ and $d$ are integers, possibly negative.% The result will be stated
%in terms of $\binom cd$, which we define by
%$$\binom cd=\begin{cases}c(c-1)\cdots(c-d+1)/d!&\text{if }d\ge0\\
%\binom c{c-d}&\text{if }d<0\text{ and }c-d\ge0\\
%0&\text{if }c<d<0.\end{cases}$$
\begin{thm}\label{thm3} If $a$ and $b$ are as in Theorem \ref{thm2}, and $c$ and $d$ are integers, then in $\Z_p$
$$\binom{ap^\infty+c}{bp^\infty+d}:=\lim_{e\to\infty}\binom{ap^e+c}{bp^e+d}=\begin{cases}\binom{ap^\infty}{bp^\infty}\binom cd&c,d\ge0\\
\binom{ap^\infty}{bp^\infty}\binom cd\frac{a-b}a&c<0\le d\\
\binom{ap^\infty}{bp^\infty}\binom c{c-d}\frac ba&c<0\le c-d\\
0&\text{otherwise.}\end{cases}$$\end{thm}
\noindent We use the standard definition that if $c\in\Z$ and $d\ge0$, then
$$\tbinom cd=c(c-1)\cdots (c-d+1)/d!.$$

These ideas arose in work of the author extending the work in \cite{D} and \cite{D2}.
\section{Proofs}\label{pfsec}
In this section, we prove the three theorems stated in Section \ref{intro}.
The main ingredient in the proof of Theorem \ref{thm1} is the following lemma.

\begin{lem}\label{lem} Let $\a$ be a positive integer which is not divisible by $p$.
Let $S$ denote the multiset consisting of the least nonnegative residues mod $p^e$ of $\u(i)$ for all $i$ satisfying $\a p^{e-1}<i\le \a p^e$.
Then every positive $p$-adic unit less than $p^e$ occurs exactly $\a$ times in $S$.\end{lem}
\begin{proof} For every $p$-adic unit $u$ in $[1,\a p^e]$, there is a unique nonnegative integer $t$ such that $\a p^{e-1}<p^tu\le\a p^e$.
This is true because the end of this interval equals $p$ times its beginning. This $u$ equals $\u(p^tu)$. Thus $S$ consists of the reductions
mod $p^e$ of all units in $[1,\a p^e]$.\end{proof}

\begin{proof}[Proof of Theorem \ref{thm1}] If $p^e>4$, the product of all $p$-adic units less than $p^e$ is congruent to $(-1)^p$ mod $p^e$.
(See, e.g., \cite[Lemma 1]{Gr}, where the argument is attributed to Gauss.) The theorem follows immediately from this and Lemma \ref{lem}, since $\u((\a p^e)!)/\u((\a p^{e-1})!)$
is the product of the numbers described in Lemma \ref{lem}.\end{proof}

\begin{proof}[Proof of Theorem \ref{thm2}] Suppose $\nu(b)=0$ and $a=\a p^k$ with $k\ge0$ and $\a=\u(a)$. Then, mod $p^e$,
\begin{eqnarray*} \u\binom{\a p^{e+k}}{bp^e}&=&\frac{\u((\a p^{e+k})!)}{\u((bp^e)!)\cdot\u(((a-b)p^e)!)}\\
&\equiv&\frac{(-1)^{p\a(e+k)}z_a}{(-1)^{pbe}z_b\cdot(-1)^{p(a-b)e}z_{a-b}}\\
&=&(-1)^{pak}\frac{z_a}{z_bz_{a-b}},\end{eqnarray*}
as claimed. Here we have used \ref{thm1} and \ref{cordef}. A similar argument works if $\nu(b)=k>0$.\end{proof}

Our proof of Theorem \ref{thm3} uses the following lemma.
\begin{lem}\label{lem2} Suppose $f$ is a function with domain $\Z\times\Z$ which satisfies Pascal's relation
\begin{equation}\label{Pas}f(n,k)=f(n-1,k)+f(n-1,k-1)\end{equation}
for all $n$ and $k$. If $f(0,d)=A\delta_{0,d}$ for all $d\in\Z$ and $f(c,0)=Ar$ for all $c<0$, then
$$f(c,d)=\begin{cases}A\binom cd&c,d\ge0\\
A\binom cdr&c<0\le d\\
A\binom c{c-d}(1-r)&c<0\le c-d\\
0&\text{otherwise.}\end{cases}$$
\end{lem}
The proof of this lemma is straightforward and omitted. It is closely related to work in \cite{HP} and \cite{S}, in which binomial coefficients are extended
to negative arguments in a similar way. However, in that case (\ref{Pas}) does not hold if $n=k=0$.

\begin{proof}[Proof of Theorem \ref{thm3}] Fix $a\ge b>0$. If $f_e(c,d):=\binom{ap^e+c}{bp^e+d}$, where $e$ is large enough that $ap^e+c>0$ and $bp^e+d>0$, then (\ref{Pas}) holds for $f_e$.
If, as $e\to\infty$, the limit exists for two terms of this version of (\ref{Pas}), then it also does  for the third. The theorem then follows from Lemma \ref{lem2} and (\ref{one}) and (\ref{two}) below, using also that
if $d<0$, then $\binom{ap^e}{bp^e+d}=\binom{ap^e}{(a-b)p^e+|d|}$, to which (\ref{one}) can be applied.

If $d>0$, then
\begin{equation}\label{one}\binom{ap^e}{bp^e+d}=\binom{ap^e}{bp^e}\frac{((a-b)p^e)\cdots((a-b)p^e-d+1)}{(bp^e+1)\cdots(bp^e+d)}\to0\end{equation}
in $\Z_p$ as $e\to\infty$, since it is $p^e$ times a factor whose $p$-exponent does not change as $e$ increases through large values.

Let $c=-m$ with $m>0$. Then
\begin{equation}\label{two}\binom{ap^e-m}{bp^e}=\binom{ap^e}{bp^e}\frac{((a-b)p^e)\cdots((a-b)p^e-m+1)}{ap^e\cdots(ap^e-m+1)}\to\binom{ap^\infty}{bp^\infty}\frac{a-b}a,\end{equation}
as $e\to\infty$, since
$$\frac{((a-b)p^e-1)\cdots((a-b)p^e-m+1)}{(ap^e-1)\cdots(ap^e-m+1)}\equiv1\mod p^{e-[\log_2(m)]}.$$
\end{proof}

\def\line{\rule{.6in}{.6pt}}

\end{document}